\documentclass[a4paper,11pt,oneside,reqno]{amsart}
\usepackage[T1]{fontenc}
\usepackage{lmodern}
\usepackage{amsmath, amssymb, amsthm, mathtools}
\usepackage{hyperref}
\usepackage[a4paper,margin=1in]{geometry}
\usepackage{enumitem}
\hypersetup{bookmarksnumbered=true}
\numberwithin{equation}{section}

\newtheorem{theorem}{Theorem}[section]
\newtheorem{introthm}{Theorem}

\newtheorem{lemma}[theorem]{Lemma}

\theoremstyle{definition}
\newtheorem{definition}[theorem]{Definition}
\theoremstyle{remark}
\newtheorem{remark}[theorem]{Remark}

\newcommand{\R}{\mathbb{R}}
\newcommand{\haus}{\mathcal{H}}
\newcommand{\Hom}{\mathrm{Hom}}
\newcommand{\G}{\mathbf{G}}
\newcommand{\loc}{\mathrm{loc}}
\newcommand{\wt}[1]{\mathord{\lVert #1 \rVert}}
\newcommand{\V}{\mathbf{V}}
\newcommand{\IV}{\mathbf{IV}}
\newcommand{\id}{\mathrm{id}}
\renewcommand{\Im}{\operatorname{Im}}
\DeclarePairedDelimiter{\abs}{\lvert}{\rvert}
\DeclarePairedDelimiter{\norm}{\lVert}{\rVert}
\DeclareMathOperator{\spt}{spt}
\DeclareMathOperator{\graph}{graph}

\begin{document}

    \title[$L^2$ normal velocity implies strong solution]{$L^2$ normal velocity implies strong solution for graphical Brakke flows}
    \author{Kotaro Motegi}
    \address{Department of Mathematics, Institute of Science Tokyo, 2-12-1 Ookayama, Meguro-ku, Tokyo 152-8551, Japan}
    \email{motegi.k.3c77@m.isct.ac.jp}

    \begin{abstract}
        We prove that if a one-parameter family of varifolds has an $L^2$ normal velocity $v$ in the sense of Brakke, and if the family is represented as the graph of a continuous function $f$ with continuous spatial derivative $\nabla f$, then $f$ has weak derivatives $\partial_t f, \nabla^2 f \in L^2$, and $v$ coincides with the usual normal velocity of the graph.
        Moreover, by combining this result with parabolic regularity theory, we show that graphical Brakke flows with forcing term in $L^{p,q}$ and $C^{0,\alpha}$ are strong and classical solutions to the forced mean curvature flow equation, respectively.
    \end{abstract}

    \maketitle

    \section{Introduction}
    \label{sec:introduction}

    A family of $k$-dimensional submanifolds $\{M_t\}_{t \geq 0}$ is called a mean curvature flow (hereafter abbreviated as MCF) if its normal velocity $v$ is equal to the mean curvature $h$ of $M_t$ at each point and time.
    More generally, we consider the MCF with forcing term, that is, a family of submanifolds $\{M_t\}_{t \geq 0}$ evolving under the motion law
    \begin{equation}
        \label{eq:intro_motion_law}
        v = h + u^\perp
    \end{equation}
    for a given ambient vector field $u$, where $u^\perp$ denotes the projection of $u$ onto the orthogonal complement of the tangent space of $M_t$.
    Given a smooth $k$-dimensional submanifold $M_0$ and a smooth vector field $u$, there exists a unique smooth MCF with forcing term $u$ starting from $M_0$, at least for a short time.
    However, even in the case $u \equiv 0$, the flow may develop singularities in finite time, such as extinction or pinching.
    To extend the MCF past singularities, Brakke introduced in his pioneering work \cite{Bra78} a generalized notion of MCF, now known as the Brakke flow, formulated in terms of varifolds in geometric measure theory.
    Within this framework, the normal velocity $v$ of a family of varifolds $\{V_t\}_{t \geq 0}$ is characterized by the following inequality:
    \begin{equation}
        \label{eq:intro_Brakke}
        \int_{\R^n} \phi(\cdot,t_2) \,d\wt{V_{t_2}} - \int_{\R^n} \phi(\cdot,t_1) \,d\wt{V_{t_1}} \leq \int_{t_1}^{t_2}\int_{\R^n} (-\phi h + \nabla \phi) \cdot v + \partial_t \phi \,d\wt{V_t}dt
    \end{equation}
    for all $0 \leq t_1 < t_2$ and $\phi \in C^1_c(\R^n \times [0,\infty);[0,\infty))$.
    Here, $\wt{V_t}$ is the weight measure of $V_t$ and $h$ is the generalized mean curvature of $V_t$ (see Section \ref{subsec:varifolds} for the definitions).
    Roughly speaking, a family $\{V_t\}_{t \geq 0}$ is called a Brakke flow with forcing term $u$ if \eqref{eq:intro_Brakke} holds for $v = h + u^\perp$.

    The present paper is concerned with regularity properties of the Brakke flow with forcing term.
    Although the Brakke flow is defined only by an inequality, it admits powerful $\varepsilon$-regularity theorems.
    In particular, Kasai and Tonegawa \cite{KT14} proved that if a Brakke flow $\{V_t\}_{t \in I}$ in an open set $U$ with forcing term $u$ is weakly close to a $k$-dimensional plane in the sense of measures, and if $u$ satisfies the dimensionally sharp integrability condition
    \begin{equation}
        \label{eq:intro_integrability}
        \norm{u}_{L^{p,q}} \coloneqq \biggl(\int_I\biggl(\int_U \abs{u(x,t)}^p \,d\wt{V_t}\biggr)^{q/p}dt\biggr)^{1/q} < \infty
    \end{equation}
    for $p, q \in [2,\infty)$ with $\zeta \coloneqq 1 - k/p - 2/q > 0$, then $\{V_t\}_{t \in I}$ can be represented as a $C^{1,\zeta}$ graph in the space-time interior.
    Here $C^{1,\zeta}$ means $C^{1,\zeta}$ in the space variables and $C^{0,(1+\zeta)/2}$ in the time variable.
    In the following, $C^{0,\alpha}$ and $C^{2,\alpha}$ should be understood in the similar manner.
    This result was subsequently extended by Stuvard and Tonegawa \cite{ST24} to regularity up to the end-time.
    Recently, De Philippis, Gasparetto, and Schulze \cite{DPGS24} gave a short proof of the $C^{1,\zeta}$-regularity for Brakke flows with bounded forcing term, using viscosity techniques.

    Since the motion law \eqref{eq:intro_motion_law} is a second-order quasilinear parabolic PDE, it is natural to expect that the graph function $f$ obtained in \cite{KT14, ST24, DPGS24} is an $L^{p,q}$-strong solution to \eqref{eq:intro_motion_law}; that is, $\partial_t f, \nabla^2 f \in L^{p,q}$, and \eqref{eq:intro_motion_law} holds almost everywhere.
    However, since the formulation of the Brakke flow is based on the inequality \eqref{eq:intro_Brakke}, standard parabolic regularity theory cannot be directly applied.
    Consequently, it remains unclear whether \eqref{eq:intro_Brakke} implies the PDE \eqref{eq:intro_motion_law} for a $C^{1,\zeta}$ graph in general, and only partial results have been obtained so far.
    If the forcing term $u$ is $\alpha$-H\"{o}lder continuous, it has been proven in \cite{Ton14, ST24} that $f \in C^{2,\alpha}$ and that \eqref{eq:intro_motion_law} holds in the classical sense.
    The proof is based on the $C^{2,\alpha}$-version of the blow-up argument, which is the central technique used in \cite{KT14}.
    For forcing terms in $L^{p,q}$, Mori, Tomimatsu, and Tonegawa \cite{MTT23} obtained an affirmative result in the codimension-one case under the additional assumption that $\partial_t f$ is a signed Radon measure.
    To establish \eqref{eq:intro_motion_law}, they first applied the Riesz representation theorem to obtain a Radon measure that turns the inequality \eqref{eq:intro_Brakke} into an equality.
    The desired motion law then follows from substituting the mollified signed distance function from the graph of $f$ into this equality.
    The additional assumption was required to control errors arising from the mollification.

    The purpose of the present paper is to prove that graphical Brakke flows with forcing term are strong solutions to \eqref{eq:intro_motion_law}, without imposing any additional assumptions.
    The following is a special case of Theorem \ref{thm:main}.

    \begin{introthm}
        \label{thm:intro_main}
        Let $\Omega \subset \R^k$ be an open set and $I \subset \R$ an open interval.
        Suppose that a family $\{V_t\}_{t \in I}$ of varifolds in $\Omega \times \R^{n-k}$ has a normal velocity $v \in L^2_\loc(\mu;\R^n)$; that is, $\{V_t\}_{t \in I}$ and $v$ satisfy \eqref{eq:intro_Brakke}, where $\mu$ denotes the Radon measure on $\Omega \times \R^{n-k} \times I$ defined by $d\mu \coloneqq d\norm{V_t}dt$.
        Assume further that $\{V_t\}_{t \in I}$ is represented as the graph of a continuous function $f \colon \Omega \times I \to \R^{n-k}$ with continuous spatial derivative $\nabla f$.
        Then $f$ has weak derivatives $\partial_t f, \nabla^2 f \in L^2_\loc(\Omega \times I)$, and $v$ coincides with the usual normal velocity of the family of graphs $\{\graph f(\cdot,t)\}_{t \in I}$.
    \end{introthm}

    More generally, Theorem \ref{thm:main} also applies to the weaker notion of MCF known as the $L^2$-flow, which is introduced by Mugnai and R\"{o}ger in \cite{MR08} (see Remark \ref{rem:L2-flow}).
    Applying Theorem \ref{thm:intro_main} to Brakke flows with forcing term, we obtain that the graph function $f$ is a strong solution to \eqref{eq:intro_motion_law}.
    Once the PDE \eqref{eq:intro_motion_law} is established, standard parabolic regularity theory yields the $L^{p,q}$-regularity of $f$, as stated below.

    \begin{introthm}
        \label{thm:intro_Lpq}
        Let $p, q \in [2,\infty)$, $R > 0$, and let $B^k_R$ be the open ball in $\R^k$ centered at the origin with radius $R$.
        Suppose that $\{V_t\}_{t \in [-R^2,0]}$ is a Brakke flow in $B^k_R \times \R^{n-k}$ with forcing term $u$ satisfying \eqref{eq:intro_integrability}, and that $\{V_t\}_{t \in [-R^2,0]}$ is represented as the graph of a continuous function $f \colon \overline{B}^k_R \times [-R^2,0] \to \R^{n-k}$ with continuous spatial derivative $\nabla f$.
        Then
        \begin{equation*}
            f \in W^{1,q}((-R^2/4,0);L^p(B^k_{R/2};\R^{n-k})) \cap L^q((-R^2/4,0);W^{2,p}(B^k_{R/2};\R^{n-k}))
        \end{equation*}
        and \eqref{eq:intro_motion_law} holds almost everywhere.
    \end{introthm}

    The precise statement is given in Theorem \ref{thm:reg_Lpq}.
    Moreover, Theorem \ref{thm:intro_main}, together with the Schauder estimates, implies the $C^{2,\alpha}$-regularity of the Brakke flow with forcing term in $C^{0,\alpha}$, thereby providing a new and concise proof of \cite[Theorem 2.3]{ST24} without relying on the blow-up argument (see Theorem \ref{thm:reg_Holder}).

    The key idea of the proof lies in the choice of test functions.
    As mentioned above, in \cite{MTT23}, the mollified signed distance function $d^\varepsilon$ is used as a test function in the equation obtained from \eqref{eq:intro_Brakke}.
    Here, we replace it with the $l$-th coordinate function $x^l$ for $l = k+1,\dots,n$.
    Since $x^l$ coincides with the $l$-th component of $f$ on its graph, substituting it as a test function gives an identity for $\partial_t(f\sqrt{g})$, where $\sqrt{g}$ denotes the area element of the graph of $f$.
    The advantage of using $x^l$ is that it is independent of time.
    Consequently, our identity does not contain a term analogous to the one involving $\partial_t d^\varepsilon$ in \cite{MTT23}, which is the main reason for the additional assumption required there.
    Essentially, this choice of test function is the same as that in \cite[Lemma 8.4]{KT14}, where it is used to derive the linearized equation of \eqref{eq:intro_motion_law}.
    Unlike in the linearization procedure, where the area element $\sqrt{g}$ can be ignored, here it must be taken into account.
    Therefore, we also derive another identity for $\partial_t\sqrt{g}$ to cancel the terms arising from $\sqrt{g}$, and then combine these identities to obtain an explicit formula for $\partial_t f$.

    The paper is organized as follows.
    Section \ref{sec:preliminaries} introduces the basic notation and definitions, and states the main theorems precisely.
    In Section \ref{sec:Proof}, we give the proof of Theorem \ref{thm:main}.
    In Section \ref{sec:Brakke_with_forcing}, we apply Theorem \ref{thm:main} to Brakke flows with forcing term and prove Theorems \ref{thm:reg_Lpq} and \ref{thm:reg_Holder}.

    \subsection*{Acknowledgements}
    The author would like to express his gratitude to his supervisor, Professor Yoshihiro Tonegawa, and to his friend, Kiichi Tashiro, for their valuable feedback and constant encouragement.

    \section{Preliminaries and main results}
    \label{sec:preliminaries}

    \subsection{Basic notation}
    \label{subsec:notatoin}

    Throughout this paper, $n$ and $k$ denote positive integers with $1 \leq k < n$.
    For $R > 0$, let $B_R$ and $B^k_R$ denote the open balls in $\R^n$ and $\R^k$ centered at the origin with radius $R$, respectively.
    We often identify $\R^k$ with $\R^k \times \{0\} \subset \R^n$.
    Let $\G(n,k)$ be the set of all $k$-dimensional subspaces of $\R^n$.
    We regard $\G(n,k)$ as a subset of $\Hom(\R^n,\R^n)$ by identifying each $S \in \G(n,k)$ with the orthogonal projection onto $S$.
    For $T \in \G(n,k)$ and $R > 0$, we define the cylinder $C_R(T) \coloneqq \{x \in \R^n : \abs{Tx} < R\}$, and the parabolic cylinder $Q_R(T) \coloneqq (B_R \cap T) \times (-R^2,0]$.
    When $T = \R^k \times \{0\}$, we write $Q^k_R$.

    Let $\Omega \subset \R^k$ be an open set, and let $I \subset \R$ be an interval.
    For $p, q \in [1,\infty]$ and a function $f$ defined on $\Omega \times I$, we define the $L^{p,q}$ norm of $f$ by
    \begin{equation*}
        \norm{f}_{L^{p,q}(\Omega \times I)} \coloneqq \biggl(\int_I\biggl(\int_\Omega \abs{f(x,t)}^p \,dx\biggr)^{q/p}dt\biggr)^{1/q}.
    \end{equation*}
    The set of all functions on $\Omega \times I$ with finite $L^{p,q}$ norm is denoted by $L^{p,q}(\Omega \times I)$.
    For $\alpha \in (0,1)$, we define the parabolic H\"{o}lder seminorms of $f$ by
    \begin{align*}
        [f]_{\alpha,\Omega \times I} &\coloneqq \sup_{(x,t),(y,s) \in \Omega \times I,\; (x,t) \neq (y,s)} \frac{\abs{f(x,t)-f(y,s)}}{\max\{\abs{x-y}^\alpha,\abs{t-s}^{\alpha/2}\}}, \\
        [f]_{1+\alpha,\Omega \times I} &\coloneqq [\nabla f]_{\alpha,\Omega \times I} + \sup_{x \in \Omega,\; t,s \in I,\; t \neq s} \frac{\abs{f(x,t)-f(x,s)}}{\abs{t-s}^{(1+\alpha)/2}}, \\
        [f]_{2+\alpha,\Omega \times I} &\coloneqq [\partial_t f]_{\alpha,\Omega \times I} + [\nabla^2 f]_{\alpha,\Omega \times I},
    \end{align*}
    where $\nabla f$ and $\nabla^2 f$ denote the spatial gradient and Hessian of $f$, respectively.
    For $j = 0, 1, 2$, we define the parabolic $C^{j,\alpha}$ norm of $f$ by
    \begin{equation*}
        \norm{f}_{j+\alpha,\Omega \times I} \coloneqq \sum_{l + 2m \leq j} \norm{\nabla^l\partial_t^m f}_{0,\Omega \times I} + [f]_{j+\alpha,\Omega \times I},
    \end{equation*}
    where $\norm{f}_{0,\Omega \times I} \coloneqq \sup_{(x,t) \in \Omega \times I} \abs{f(x,t)}$.
    The set of all functions on $\Omega \times I$ with finite parabolic $C^{j,\alpha}$ norm is denoted by $C^{j,\alpha}(\overline{\Omega \times I})$.

    \subsection{Varifolds}
    \label{subsec:varifolds}

    We recall some notions related to varifolds and refer to \cite{All72, Sim83} for more details.
    Let $U \subset \R^n$ be an open set, and define $G_k(U) \coloneqq U \times \G(n,k)$.
    A $k$-varifold in $U$ is a Radon measure on $G_k(U)$, and the set of all $k$-varifolds in $U$ is denoted by $\V_k(U)$.
    For $V \in \V_k(U)$, its weight measure $\wt{V}$ is the Radon measure on $U$ defined by 
    \begin{equation*}
        \wt{V}(\phi) \coloneqq \int_{G_n(U)} \phi(x) \,dV(x,S)
    \end{equation*}
    for every $\phi \in C^0_c(U)$.
    We say that $V \in \V_k(U)$ is integral if there exist an $\haus^k$-measurable, countably $k$-rectifiable set $M \subset U$ and a non-negative, integer-valued, locally $\haus^k$-integrable function $\theta$ on $M$ such that
    \begin{equation*}
        V(\phi) = \int_M \phi(x,T_xM)\theta(x) \,d\haus^k(x)
    \end{equation*} 
    for every $\phi \in C^0_c(G_k(U))$, where $T_xM$ denotes the approximate tangent space of $M$ at $x$, which exists for $\haus^k$-a.e. $x \in M$.
    The set of all integral $k$-varifolds in $U$ is denoted by $\IV_k(U)$.
    When $\theta = 1$, we say that $V$ is of unit density, and we write $V = \abs{M}$.

    The first variation of $V \in \V_k(U)$ is the linear functional $\delta V$ on $C^1_c(U;\R^n)$ defined by
    \begin{equation*}
        \delta V(g) \coloneqq \int_{G_k(U)} \nabla g(x) \cdot S \,dV(x,S)
    \end{equation*}
    for every $g \in C^1_c(U;\R^n)$.
    When $\delta V$ extends to a continuous linear functional on $C^0_c(U;\R^n)$, we say that $V$ has locally bounded first variation, and denote the corresponding total variation measure by $\wt{\delta V}$.
    If $\wt{\delta V}$ is absolutely continuous with respect to $\wt{V}$, there exists a unique vector field $h(V,\cdot) \in L^1_\loc(\wt{V};\R^n)$ such that
    \begin{equation*}
        \delta V(g) = -\int_U g(x) \cdot h(V,x) \,d\wt{V}(x)
    \end{equation*}
    for all $g \in C^1_c(U;\R^n)$.
    We call $h(V,\cdot)$ the generalized mean curvature of $V$.

    \subsection{Generalized notion of normal velocity}
    \label{subsec:velocity}

    In this subsection, we introduce a generalized notion of normal velocity for varifolds and define Brakke flows with forcing term.
    Let $U \subset \R^n$ be an open set, and let $I \subset \R$ be an interval. 
    
    \begin{definition}
        \label{def:velocity_Brakke}
        Suppose a family $\{V_t\}_{t \in I} \subset \V_k(U)$ satisfies the following conditions:
        \begin{enumerate}[label=(\roman*)]
            \item \label{itm:v_integrality}For a.e. $t \in I$, $V_t \in \IV_k(U)$.
            \item \label{itm:v_finite_mass}For any $\tilde{U} \subset\subset U$ and $\tilde{I} \subset\subset I$, we have $\sup_{t \in \tilde{I}}\wt{V_t}(\tilde{U}) < \infty$.
            \item \label{itm:v_first_variation}For a.e. $t \in I$, $V_t$ has locally bounded first variation and $\wt{\delta V_t} \ll \wt{V_t}$.
            \item \label{itm:v_mean_curv}For any $\tilde{U} \subset\subset U$ and $\tilde{I} \subset\subset I$, we have
                \begin{equation*}
                    \int_{\tilde{I}}\int_{\tilde{U}} \abs{h(V_t,x)}^2 \,d\wt{V_t}(x)dt < \infty.
                \end{equation*}
        \end{enumerate}
        Let $\mu$ be the Radon measure on $U \times I$ defined by $d\mu \coloneqq d\wt{V_t}dt$.
        A vector field $v \in L^2_\loc(\mu;\R^n)$ is called a generalized normal velocity for $\{V_t\}_{t \in I}$ if the following conditions hold:
        \begin{enumerate}[resume*]
            \item \label{itm:v_perp}For a.e. $t \in I$ and $V_t$-a.e. $(x,S) \in G_k(U)$, we have $Sv(x,t) = 0$.
            \item \label{itm:v_Brakke_ineq}For all $t_1, t_2 \in I$ with $t_1 < t_2$ and $\phi \in C^1_c(U \times I;[0,\infty))$, we have
                \begin{multline}
                    \label{eq:Brakke_ineq}
                    \int_U \phi(x,t_2) \,d\wt{V_{t_2}}(x) - \int_U \phi(x,t_1) \,d\wt{V_{t_1}}(x) \\
                    \leq \int_{t_1}^{t_2}\int_U (-\phi(x,t)h(V_t,x) + \nabla \phi(x,t)) \cdot v(x,t) + \partial_t \phi(x,t) \,d\wt{V_t}(x)dt.
                \end{multline}
        \end{enumerate}
    \end{definition}

    If $\{M_t\}_{t \in I}$ is a smooth family of $k$-dimensional submanifolds, the usual normal velocity is the unique generalized normal velocity for $\{\abs{M_t}\}_{t \in I}$ and \eqref{eq:Brakke_ineq} holds with equality.
    See \cite[Chapter 2]{Ton19} for details on this definition.

    \begin{definition}
        \label{def:Brakke_with_forcing_term}
        Let $\{V_t\}_{t \in I} \subset \V_k(U)$ and $u \in L^2_\loc(\mu,\R^n)$, where $\mu$ is as in Definition \ref{def:velocity_Brakke}.
        We say that $\{V_t\}_{t \in I}$ is a Brakke flow with forcing term $u$ if it satisfies \ref{itm:v_integrality}--\ref{itm:v_mean_curv} of Definition \ref{def:velocity_Brakke}, and if $v(x,t) \coloneqq h(V_t,x) + u^\perp(x,t)$ is a generalized normal velocity of $\{V_t\}_{t \in I}$.
        Here $u^\perp(x,t)$ denotes the projection of $u(x,t)$ onto the orthogonal complement of the approximate tangent space of $V_t$ at $x$.
    \end{definition}

    Note that $v = h + u^\perp$ automatically satisfies \ref{itm:v_perp} of Definition \ref{def:velocity_Brakke} by Brakke's perpendicularity theorem \cite[5.8]{Bra78}.

    \subsection{Main results}
    \label{subsec:main_results}

    The following is the main theorem of this paper.

    \begin{theorem}
        \label{thm:main}
        Let $T \in \G(n,k)$ be a $k$-dimensional subspace, $\Omega \subset T$ an open set, and let $I \subset \R$ be an open interval.
        Suppose that a function $f \colon \Omega \times I \to T^\perp$ satisfies the following conditions:
        \begin{enumerate}[label=(\arabic*)]
            \item \label{itm:main_conti_f}The function $f \in C^0(\Omega \times I;T^\perp)$.
            \item \label{itm:main_conti_nabla_f}For all $t \in I$, $f(\cdot,t)$ is differentiable on $\Omega$, and its spatial derivative $\nabla f \in C^0(\Omega \times I;\Hom(T;T^\perp))$.
        \end{enumerate}
        Let $V_t$ denote the unit density $k$-varifold in $U \coloneqq \{x \in \R^n : Tx \in \Omega\}$ defined by $V_t \coloneqq \abs{\graph f(\cdot,t)}$, and let $\mu$ be the Radon measure on $U \times I$ defined by $d\mu \coloneqq d\wt{V_t}dt$.
        Suppose further that the family $\{V_t\}_{t \in I}$ and a vector field $v \in L^2_\loc(\mu;\R^n)$ satisfy the following conditions:
        \begin{enumerate}[resume*]
            \item \label{itm:main_delta_V}For a.e. $t \in I$, $V_t$ has locally bounded first variation and $\wt{\delta V_t} \ll \wt{V_t}$.
            \item \label{itm:main_mean_curv}For any $\tilde{U} \subset\subset U$ and $\tilde{I} \subset\subset I$, the generalized mean curvature $h(V_t,\cdot)$ of $V_t$ satisfies
                \begin{equation}
                    \label{eq:main_mean_curv}
                    \int_{\tilde{I}}\int_{\tilde{U}} \abs{h(V_t,x)}^2 \,d\wt{V_t}dt < \infty.
                \end{equation}
            \item \label{itm:main_L2_velocity}For any $\tilde{U} \subset\subset U$ and $\tilde{I} \subset\subset I$, there exists a constant $C_{\tilde{U},\tilde{I}} > 0$ such that for all $\phi \in C^1_c(\tilde{U} \times \tilde{I})$,
                \begin{equation}
                    \label{eq:main_L2_velocity}
                    \abs*{\int_I\int_U \nabla \phi(x,t) \cdot v(x,t) + \partial_t \phi(x,t) \,d\wt{V_t}dt} \leq C_{\tilde{U},\tilde{I}}\norm{\phi}_{C^0(\tilde{U} \times \tilde{I})}.
                \end{equation}
        \end{enumerate}
        Then $f$ has weak derivatives
        \begin{equation*}
            \partial_t f \in L^2_\loc(\Omega \times I;T^\perp) \qquad \text{and} \qquad \nabla^2 f \in L^2_\loc(\Omega \times I;\Hom(T \otimes T,T^\perp)).
        \end{equation*}
        Moreover, the following identity holds for a.e. $(x,t) \in \Omega \times I$:
        \begin{equation}
            \label{eq:main_t_derivative}
            \partial_t f(x,t) = T^\perp v(x+f(x,t),t) - \nabla_{Tv(x+f(x,t),t)} f(x,t).
        \end{equation}
    \end{theorem}

    \begin{remark}
        \label{rem:L2-flow}
        Assumption \ref{itm:main_L2_velocity} is equivalent to saying that $v$ is, locally in space-time, a generalized normal velocity in the sense of the $L^2$-flow defined by Mugnai and R\"{o}ger in \cite{MR08}.
        A generalized normal velocity defined in Definition \ref{def:velocity_Brakke} satisfies assumption \ref{itm:main_L2_velocity}.
        Indeed, \eqref{eq:Brakke_ineq} implies that the linear functional $L$ defined by
        \begin{equation*}
            L(\phi) \coloneqq \int_I\int_U (-\phi(x,t)h(V_t,x) + \nabla \phi(x,t)) \cdot v(x,t) + \partial_t\phi(x,t) \,d\wt{V_t}(x)dt
        \end{equation*}
        for $\phi \in C^1_c(U \times I)$ is non-negative.
        Hence, by \cite[Theorem 1.39]{EG92}, $L$ is a (non-negative) Radon measure on $U \times I$; in particular, \eqref{eq:main_L2_velocity} holds.
    \end{remark}

    \begin{remark}
        \label{rem:identity}
        The identity \eqref{eq:main_t_derivative} implies that $v^\perp$ is the usual normal velocity of the family of graphs $\{\graph f(\cdot,t)\}_{t \in I}$.
        To see this, set $S(x,t) \coloneqq T_{x+f(x,t)}\graph f(\cdot,t) \in \G(n,k)$.
        Then, since $S(x,t) = \Im(T + \nabla f(x,t))$, we have $S(x,t)^\perp T = - S(x,t)^\perp \nabla f(x,t)$.
        Hence we obtain
        \begin{align*}
            S(x,t)^\perp \partial_t f(x,t) &= S(x,t)^\perp T^\perp v(x+f(x,t),t) - S(x,t)^\perp \nabla_{Tv(x+f(x,t),t)} f(x,t) \\
            &= S(x,t)^\perp T^\perp v(x+f(x,t),t) + S(x,t)^\perp Tv(x+f(x,t),t) \\
            &= S(x,t)^\perp v(x+f(x,t),t).
        \end{align*}
    \end{remark}
    
    We present two applications of Theorem \ref{thm:main} to Brakke flows with forcing term.
    The first concerns the case where the forcing term belongs to $L^{p,q}$.

    \begin{theorem}
        \label{thm:reg_Lpq}
        Let $R > 0$, $T \in \G(n,k)$, and $p, q \in [2,\infty)$.
        Suppose that a family of $k$-varifolds $\{V_t\}_{t \in [-R^2,0]} \subset \V_k(C_R(T))$, a family of vector fields $\{u(\cdot,t)\}_{t \in [-R^2,0]}$ on $C_R(T)$, and a function $f \in C^0([-R^2,0];C^1(\overline{B_R \cap T};T^\perp))$ satisfy the following conditions:
        \begin{enumerate}[label=(\arabic*)]
            \item \label{itm:Lpq_integrability}The vector field $u$ satisfies
                \begin{equation*}
                    \norm{u}_{L^{p,q}} \coloneqq \biggl(\int_{-R^2}^0\biggl(\int_{C_R(T)} \abs{u(x,t)}^p \,d\wt{V_t}(x)\biggr)^{q/p}\,dt\biggr)^{1/q} < \infty.
                \end{equation*}
            \item \label{itm:Lpq_Brakke}The family $\{V_t\}_{t \in [-R^2,0]}$ is a Brakke flow in $C_R(T)$ with forcing term $u$.
            \item \label{itm:graph}For all $t \in [-R^2,0]$, we have $V_t = \abs{\graph f(\cdot,t)}$.
        \end{enumerate}
        Then we have
        \begin{equation}
            \label{eq:Lpq_class}
            f \in W^{1,q}((-R^2/4,0);L^p(B_{R/2} \cap T;T^\perp)) \cap L^q((-R^2/4,0);W^{2,p}(B_{R/2} \cap T;T^\perp))
        \end{equation}
        and
        \begin{equation}
            \label{eq:Lpq_est}
            \norm{\partial_t f}_{L^{p,q}(Q_{R/2}(T))} + \norm{\nabla^2 f}_{L^{p,q}(Q_{R/2}(T))} \leq C\bigl(R^{-2}\norm{f}_{L^{p,q}(Q_R(T))} + \norm{u}_{L^{p,q}}\bigr),
        \end{equation}
        where $C > 0$ is a constant depending only on $n$, $k$, $p$, $q$, $\norm{\nabla f}_{0,Q_R(T)}$, and the modulus of continuity of $\nabla f$.
        Moreover, $f$ satisfies the motion law \eqref{eq:intro_motion_law} a.e. on $Q_{R/2}(T)$.
    \end{theorem}

    \begin{remark}
        \label{rem:graphicality}
        Suppose that $\{V_t\}_{t \in [-R^2,0]}$ and $\{u(\cdot,t)\}_{t \in [-R^2,0]}$ satisfy assumptions \ref{itm:Lpq_integrability} and \ref{itm:Lpq_Brakke} for $p, q$ with $\zeta \coloneqq 1-k/p-2/q > 0$.
        If, in addition, $\{V_t\}_{t \in [-R^2,0]}$ is weakly close to a multiplicity-one $k$-plane in the sense of varifolds, then there exists a parabolic $C^{1,\zeta}$ function $f$ satisfying assumption \ref{itm:graph} in a smaller parabolic cylinder (see \cite[Theorem 2.2]{ST24} for the precise statement).
        Furthermore, $f$ satisfies a suitable $C^{1,\zeta}$ estimate; hence, in this case, the constant $C > 0$ in \eqref{eq:Lpq_est} can be chosen independently of $f$.
    \end{remark}

    The second concerns Brakke flows with H\"{o}lder continuous forcing term.
    The following theorem was proven in \cite[Theorem 2.3]{ST24} by using the blow-up argument.
    Here, we give a new and concise proof using Theorem \ref{thm:main}.

    \begin{theorem}
        \label{thm:reg_Holder}
        Let $R > 0$, $T \in \G(n,k)$, and $\alpha \in (0,1)$.
        Suppose that a family of $k$-varifolds $\{V_t\}_{t \in [-R^2,0]} \subset \V_k(C_R(T))$, a vector field $u \in C^{0,\alpha}(\overline{C}_R(T) \times [-R^2,0];\R^n)$, and a function $f \in C^{1,\alpha}(\overline{Q}_R(T);T^\perp)$ satisfy assumptions \ref{itm:Lpq_Brakke} and \ref{itm:graph} of Theorem \ref{thm:reg_Lpq}.
        Then we have $f \in C^{2,\alpha}(\overline{Q}_{R/2}(T);T^\perp)$ and
        \begin{multline}
            \label{eq:Holder_est}
            \norm{\partial_t f}_{0,Q_{R/2}(T)} + \norm{\nabla^2 f}_{0,Q_{R/2}(T)} + R^\alpha\bigl([\partial_t f]_{\alpha,Q_{R/2}(T)} + [\nabla^2 f]_{\alpha,Q_{R/2}(T)}\bigr) \\
            \leq C\bigl(R^{-2}\norm{f}_{0,Q_R(T)} + \norm{u}_{0,C_R(T) \times (-R^2,0]} + R^\alpha [u]_{\alpha,C_R(T) \times (-R^2,0]}\bigr),
        \end{multline}
        where $C > 0$ is a constant depending only on $n$, $k$, $\alpha$, $\norm{\nabla f}_{0,Q_R(T)}$, and $R^\alpha[f]_{1+\alpha,Q_R(T)}$.
        Moreover, $f$ is a classical solution to \eqref{eq:intro_motion_law}.
    \end{theorem}

    \section{Proof of Theorem \ref{thm:main}} 
    \label{sec:Proof}

    This section is devoted to the proof of Theorem \ref{thm:main}.
    We first prove the square integrability of the Hessian of $f$.
    
    \begin{lemma}
        \label{lem:Hess_L2}
        Under the assumptions of Theorem \ref{thm:main}, $\nabla^2 f \in L^2_\loc(\Omega \times I;\Hom(T \otimes T;T^\perp))$.
    \end{lemma}

    \begin{proof}
        Without loss of generality, we may assume $T = \R^k \times \{0\}$.
        By assumption \ref{itm:main_mean_curv} of Theorem \ref{thm:main}, we have $h(V_t,\cdot) \in L^2_\loc(\wt{V_t};\R^n)$ for a.e. $t \in I$.
        For such $t \in I$, it follows from the definition of $h(V_t,\cdot) = (h^1(V_t,\cdot),\dots,h^n(V_t,\cdot))$ and the area formula that $f(\cdot,t) = (f^{k+1}(\cdot,t),\dots,f^n(\cdot,t))$ is a weak solution to the prescribed mean curvature system
        \begin{equation}
            \label{eq:Hess_mss}
            \sum_{i,j=1}^k \partial_i\Big(\sqrt{g(\nabla f(\cdot,t))}g^{ij}(\nabla f(\cdot,t))\partial_j f^a(\cdot,t)\Big) = \sqrt{g(\nabla f(\cdot,t))}h^a(V_t,\cdot + f(\cdot,t))
        \end{equation}
        for $a = k+1,\dots,n$.
        Here, $g^{ij}$ and $g$ are defined by
        \begin{equation}
            \label{eq:Hess_metric}
            (g_{ij}(P)) \coloneqq \id_{\R^k} + P^\ast P, \qquad (g^{ij}(P)) \coloneqq (g_{ij}(P))^{-1}, \qquad g(P) \coloneqq \det(g_{ij}(P))
        \end{equation}
        for $P \in \Hom(\R^k,\R^{n-k})$.
        The system \eqref{eq:Hess_mss} satisfies the Legendre--Hadamard condition; more precisely, the following inequality holds:
        \begin{equation}
            \label{eq:Hess_LH}
            \sum_{i,j,l=1}^k\sum_{a,b=1}^{n-k}\partial_{P_j^b}\Big(\sqrt{g(P)}g^{il}(P)P_l^a\Big)\xi_i\xi_j\eta^a\eta^b \geq \frac{\sqrt{g(P)}}{(1+\abs{P}^2)^2}\abs{\xi}^2\abs{\eta}^2
        \end{equation}
        for all $P = (P_i^a) \in \Hom(\R^k,\R^{n-k})$, $\xi = (\xi_1,\dots,\xi_k) \in \R^k$, and $\eta = (\eta^1,\dots,\eta^{n-k}) \in \R^{n-k}$.
        Fix $\Omega' \subset\subset \Omega'' \subset\subset \Omega$ and $I' \subset\subset I$.
        Since $\nabla f(\cdot,t)$ is bounded on $\Omega''$ by assumption \ref{itm:main_conti_nabla_f} of Theorem \ref{thm:main}, we have
        \begin{equation}
            \frac{\sqrt{g(\nabla f(\cdot,t))}}{(1 + \abs{\nabla f(\cdot,t)}^2)^2} \geq \frac{1}{(1 + \norm{\nabla f(\cdot,t)}_{0,\Omega''}^2)^{3/2}} > 0
        \end{equation}
        on $\Omega''$.
        Hence, by a difference quotient argument, we obtain
        \begin{equation}
            \label{eq:Hess_2nd_derivative}
            \norm{\nabla^2 f(\cdot,t)}_{L^2(\Omega')} \leq C(\norm{\nabla f(\cdot,t)}_{L^2(\Omega'')} + \norm{h(V_t,\cdot)}_{L^2(U'',\wt{V_t})}),
        \end{equation}
        where $U'' \coloneqq \Omega'' \times \R^{n-k}$ and $C > 0$ is a constant depending only on $n$, $k$, $\Omega'$, $\Omega''$, $\norm{\nabla f(\cdot,t)}_{0,\Omega''}$, and the modulus of continuity of $\nabla f(\cdot,t)$ on $\Omega''$.
        Integrating \eqref{eq:Hess_2nd_derivative} over $I'$ yields
        \begin{equation*}
            \norm{\nabla^2 f}_{L^2(\Omega' \times I')} \leq C'(\norm{\nabla f}_{L^2(\Omega'' \times I')} + \norm{h}_{L^2(U'' \times I',\mu)}),
        \end{equation*}
        where $C' > 0$ is a constant depending only on $n$, $k$, $\Omega'$, $\Omega''$, $I'$, $\norm{\nabla f}_{0,\Omega'' \times I'}$, and the modulus of continuity of $\nabla f$ on $\Omega'' \times I'$.
        By assumptions \ref{itm:main_conti_f} and \ref{itm:main_mean_curv} of Theorem \ref{thm:main}, the right-hand side of the inequality above is finite, which implies that $\nabla^2 f$ is locally square integrable.
    \end{proof}

    We now turn to the proof of Theorem \ref{thm:main}.

    \begin{proof}[Proof of Theorem \ref{thm:main}]
        By Lemma \ref{lem:Hess_L2} and the fact that $v \in L^2_\loc(\mu;\R^n)$, it suffices to prove \eqref{eq:main_t_derivative}.
        Assumption \ref{itm:main_L2_velocity} implies that the linear functional $L$ on $C^1_c(U \times I)$, defined by
        \begin{equation*}
            L(\phi) \coloneqq \int_I\int_U \nabla \phi \cdot v + \partial_t \phi \,d\wt{V_t}dt
        \end{equation*}
        for $\phi \in C^1_c(U \times I)$, extends to a continuous linear functional on $C^0_c(U \times I)$.
        Hence, by the Riesz representation theorem, there exists a signed Radon measure $\xi$ on $U \times I$ such that
        \begin{equation}
            \label{eq:xi}
            L(\phi) = \int_{U \times I} \phi \,d\xi
        \end{equation}
        for all $\phi \in C^1_c(U \times I)$.
        Furthermore, since $L(\phi) = 0$ for any $\phi \in C^1_c(U \times I)$ with $\spt \phi \cap \graph f = \emptyset$, it follows that
        \begin{equation}
            \label{eq:xi_supp}
            \spt \xi \subset \graph f.
        \end{equation}

        Next, we derive two identities for $\partial_t(f\sqrt{g})$ and $\partial_t \sqrt{g}$, respectively, where $\sqrt{g} \coloneqq \sqrt{g(\nabla f)}$ is as in \eqref{eq:Hess_metric}, and then combine them to obtain \eqref{eq:main_t_derivative}.
        Let $\psi \in C^1_c(\Omega \times I;T^\perp)$.
        Substituting $\phi(x,t) \coloneqq \psi(Tx,t) \cdot T^\perp x$ into \eqref{eq:xi}, we obtain
        \begin{multline*}
            \int_I\int_U (\nabla_{Tv(x,t)}\psi(Tx,t) \cdot T^\perp x + \psi(Tx,t) \cdot T^\perp v(x,t) + \partial_t \psi(Tx,t) \cdot T^\perp x \,d\wt{V_t}dt \\
            = \int_{U \times I} \psi(Tx,t) \cdot T^\perp x \,d\xi(x,t).
        \end{multline*}
        By applying the area formula and using \eqref{eq:xi_supp}, the identity above becomes
        \begin{equation}
            \label{eq:t_derivative_fg}
            -\int_I\int_\Omega \partial_t \psi \cdot f\sqrt{g} \,dxdt = \int_I\int_\Omega \nabla_{T\tilde{v}}\psi \cdot f\sqrt{g} + \psi \cdot T^\perp\tilde{v}\sqrt{g} \,dxdt - \int_{\Omega \times I} \psi \cdot f \,d\tilde{\xi},
        \end{equation}
        where $\tilde{v}(x,t) \coloneqq v(x+f(x,t),t)$, and $\tilde{\xi} \coloneqq (T \times \id_I)_\sharp \xi$ denotes the pushforward Radon measure on $\Omega \times I$.
        Similarly, for any $\tilde{\psi} \in C^1_c(\Omega \times I)$, substituting $\phi(x,t) \coloneqq \tilde{\psi}(Tx,t)$ into \eqref{eq:xi} yields
        \begin{equation}
            \label{eq:t_derivative_g}
            -\int_I\int_\Omega \partial_t \tilde{\psi} \sqrt{g} \,dxdt = \int_I\int_\Omega \nabla\tilde{\psi} \cdot T\tilde{v}\sqrt{g} \,dxdt - \int_{\Omega \times I} \tilde{\psi} \,d\tilde{\xi}.
        \end{equation}
        We define $L_1(\psi)$ and $L_2(\psi)$ by the right-hand sides of \eqref{eq:t_derivative_fg} and \eqref{eq:t_derivative_g}, respectively.
        Note that $L_1$ and $L_2$ uniquely extend to continuous linear functionals on $L^2(I;W^{1,2}(\Omega;T^\perp)) \cap C^0_c(\Omega \times I;T^\perp)$ and $L^2(I;W^{1,2}(\Omega)) \cap C^0_c(\Omega \times I)$, respectively.

        If we assume that $f$ and $v$ are smooth in space-time, then the desired identity \eqref{eq:main_t_derivative} follows from \eqref{eq:t_derivative_fg}, \eqref{eq:t_derivative_g}, and the quotient rule for derivatives.
        To justify this argument, we employ a mollifier.
        Let $\rho$ be a smooth, radially symmetric function on $T \times \R$ satisfying
        \begin{equation*}
            \spt \rho \subset (B_1 \cap T) \times (-1,1), \qquad \int_{T \times \R} \rho(x,t) \,dxdt = 1.
        \end{equation*}
        For each $\varepsilon > 0$, define $\rho_\varepsilon(x,t) \coloneqq \varepsilon^{-k-1}\rho(x/\varepsilon,t/\varepsilon)$ so that $\int_{T \times \R} \rho_\varepsilon \,dxdt = 1$.
        Hereafter, we denote by a subscript $\varepsilon$ the convolution of a function with $\rho_\varepsilon$.
        Fix $\tilde{\Omega} \subset\subset \Omega$ and $\tilde{I} \subset\subset I$, and let $\psi \in C^1_c(\tilde{\Omega} \times \tilde{I};T^\perp)$.
        Since $f$ and $g$ are continuous on $\Omega \times I$ by assumptions \ref{itm:main_conti_f} and \ref{itm:main_conti_nabla_f}, and since $g \geq 1$, we have
        \begin{equation}
            \label{eq:unif_conv}
            (f\sqrt{g})_\varepsilon \longrightarrow f\sqrt{g}, \quad \frac{1}{(\sqrt{g})_\varepsilon} \longrightarrow \frac{1}{\sqrt{g}} \qquad \text{uniformly on compact subsets of } \Omega \times I.
        \end{equation}
        Hence, the distributional time derivative $\partial_t f$ can be expressed as
        \begin{equation}
            \label{eq:limit_1}
            \begin{split}
                &-\int_I\int_\Omega \partial_t \psi \cdot f \,dxdt = -\lim_{\varepsilon \to 0} \int_{\tilde{I}}\int_{\tilde{\Omega}} \partial_t \psi \cdot \frac{(f\sqrt{g})_\varepsilon}{(\sqrt{g})_\varepsilon} \,dxdt \\
                &= -\lim_{\varepsilon \to 0} \int_{\tilde{I}}\int_{\tilde{\Omega}} \partial_t\biggl(\frac{\psi}{(\sqrt{g})_\varepsilon}\biggr) \cdot (f\sqrt{g})_\varepsilon + \psi \cdot (f\sqrt{g})_\varepsilon \frac{\partial_t(\sqrt{g})_\varepsilon}{((\sqrt{g})_\varepsilon)^2} \,dxdt \\
                &= -\lim_{\varepsilon \to 0} \int_{\tilde{I}}\int_{\tilde{\Omega}} \partial_t\biggl(\frac{\psi}{(\sqrt{g})_\varepsilon}\biggr) \cdot (f\sqrt{g})_\varepsilon - \partial_t\biggl(\frac{\psi \cdot (f\sqrt{g})_\varepsilon}{((\sqrt{g})_\varepsilon)^2}\biggr) (\sqrt{g})_\varepsilon \,dxdt \\
                &= \lim_{\varepsilon \to 0} \biggl[L_1\biggl(\biggl(\frac{\psi}{(\sqrt{g})_\varepsilon}\biggr)_\varepsilon\biggr) - L_2\biggl(\biggl(\frac{\psi \cdot (f\sqrt{g})_\varepsilon}{((\sqrt{g})_\varepsilon)^2}\biggr)_\varepsilon\biggr)\biggr].
            \end{split}
        \end{equation}
        From assumptions \ref{itm:main_conti_f}, \ref{itm:main_conti_nabla_f}, and Lemma \ref{lem:Hess_L2}, it follows that $\nabla \sqrt{g} \in L^2_\loc(\Omega \times I;T)$.
        Combined with \eqref{eq:unif_conv}, this gives that
        \begin{alignat}{2}
            \label{eq:L2_conv_fg}
            \nabla\biggl(\frac{\psi}{(\sqrt{g})_\varepsilon}\biggr)_\varepsilon &\longrightarrow \nabla\biggl(\frac{\psi}{\sqrt{g}}\biggr) \qquad & &\text{in } L^2(\Omega \times I;\Hom(T;T^\perp)), \\
            \label{eq:L2_conv_g}
            \nabla\biggl(\frac{\psi \cdot (f\sqrt{g})_\varepsilon}{((\sqrt{g})_\varepsilon)^2}\biggr)_\varepsilon &\longrightarrow \nabla\biggl(\frac{\psi \cdot f}{\sqrt{g}}\biggr) \qquad & &\text{in } L^2(\Omega \times I;T).
        \end{alignat}
        Therefore, we conclude from \eqref{eq:unif_conv}--\eqref{eq:L2_conv_g} that
        \begin{align*}
            -\int_I\int_\Omega \partial_t\psi \cdot f \,dxdt &= L_1\biggl(\frac{\psi}{\sqrt{g}}\biggr) - L_2\biggl(\frac{\psi \cdot f}{\sqrt{g}}\biggr) \\
            &= \int_I\int_\Omega \nabla_{T\tilde{v}}\biggl(\frac{\psi}{\sqrt{g}}\biggr) \cdot f\sqrt{g} + \frac{\psi}{\sqrt{g}} \cdot T^\perp\tilde{v}\sqrt{g} \,dxdt - \int_{\Omega \times I} \frac{\psi}{\sqrt{g}} \cdot f \,d\tilde{\xi} \\
            &\quad -\int_I\int_\Omega \nabla\biggl(\frac{\psi \cdot f}{\sqrt{g}}\biggr) \cdot T\tilde{v}\sqrt{g} \,dxdt + \int_{\Omega \times I} \frac{\psi \cdot f}{\sqrt{g}} \,d\tilde{\xi} \\
            &= \int_I\int_\Omega \psi \cdot (T^\perp\tilde{v} - \nabla_{T\tilde{v}}f) \,dxdt,
        \end{align*}
        which implies \eqref{eq:main_t_derivative}.
    \end{proof}
    
    \section{Regularity for Brakke flows with forcing term}
    \label{sec:Brakke_with_forcing}

    In this section, we prove Theorems \ref{thm:reg_Lpq} and \ref{thm:reg_Holder}.
    Hereafter, we adopt Einstein's summation convention, namely, the summation symbol is omitted for repeated indices.
    We begin by deriving the parabolic PDE satisfied by graphical Brakke flows with forcing term.
    Suppose that $\{V_t\}_{t \in [-1,0]}$, $\{u(\cdot,t)\}_{t \in [-1,0]}$, and $f \in C^0([-1,0];C^1(\overline{B}^k_1;\R^{n-k}))$ satisfy assumptions \ref{itm:Lpq_integrability}--\ref{itm:graph} of Theorem \ref{thm:reg_Lpq} with $R = 1$, $T = \R^k \times \{0\}$, and $p = q = 2$.
    From Theorem \ref{thm:main} and Remark \ref{rem:L2-flow}, it follows that
        \begin{equation*}
            \partial_t f \in L^2_\loc(Q^k_1;\R^{n-k}), \quad \nabla^2 f \in L^2_\loc(Q^k_1;\R^{k^2(n-k)}),
        \end{equation*}
        and
        \begin{multline}
            \label{eq:MCF_eq}
            \partial_t f^a(x,t) = h^a(V_t,x+f(x,t)) + (u^\perp)^a(x+f(x,t),t) \\
            - \bigl(h^j(V_t,x+f(x,t))+(u^\perp)^j(x+f(x,t),t)\bigr)\partial_jf^a(x,t)
        \end{multline}
        for a.e. $(x,t) \in Q^k_1$ and for $a = k+1, \dots, n$, where superscripts indicate the corresponding components of vectors in $\R^n$.
        Furthermore, by Remark \ref{rem:identity}, \eqref{eq:MCF_eq} implies the motion law \eqref{eq:intro_motion_law}.
        Define
        \begin{equation*}
            (g^{ij}(x,t)) \coloneqq (\delta_{ij} + \partial_i f(x,t) \cdot \partial_j f(x,t))^{-1}, \qquad g(x,t) \coloneqq \det(\delta_{ij} + \partial_i f(x,t) \cdot \partial_j f(x,t)),
        \end{equation*}
        and
        \begin{equation*}
            U^a(x,t) \coloneqq (u^\perp)^a(x+f(x,t),t) - (u^\perp)^j(x+f(x,t),t)\partial_jf^a(x,t)
        \end{equation*}
        for $a = k+1,\dots,n$.
        Applying the area formula to the definition of $h(V_t,\cdot)$, we obtain
        \begin{alignat*}{2}
            h^j(V_t,x + f(x,t)) &= \frac{1}{\sqrt{g(x,t)}}\partial_i\Big(\sqrt{g(x,t)}g^{ij}(x,t)\Big) \qquad & &\text{for } j = 1,\dots,k, \\
            h^a(V_t,x+f(x,t)) &= \frac{1}{\sqrt{g(x,t)}}\partial_i\Big(\sqrt{g(x,t)}g^{ij}(x,t)\partial_jf^a(x,t)\Big) \qquad & &\text{for } a = k+1,\dots,n.
        \end{alignat*}
        Substituting these identities into \eqref{eq:MCF_eq} yields the PDE
        \begin{equation}
            \label{eq:PDE}
            \begin{split}
                \partial_t f^a &= \frac{1}{\sqrt{g}}\partial_i(\sqrt{g}g^{ij}\partial_j f^a) - \frac{1}{\sqrt{g}}\partial_i(\sqrt{g}g^{ij})\partial_j f^a + U^a \\
                &= g^{ij}\partial_{ij}f^a + U^a.
            \end{split}
        \end{equation}
        Since $\nabla f$ is uniformly continuous on $Q^k_1$, the coefficients $g^{ij}$ are bounded, uniformly continuous, and uniformly elliptic.
        In fact, the following estimates hold:
        \begin{gather}
            \label{eq:g_bound}
            \norm{g^{ij}}_{0,Q^k_1} \leq 1, \\
            \label{eq:g_ellipticity}
            g^{ij}\xi_i\xi_j \geq \frac{1}{1 + \norm{\nabla f}_{0,Q^k_1}}\abs{\xi}^2 \quad \text{for all } \xi = (\xi_1,\dots,\xi_k) \in \R^k.
        \end{gather}

        Once the parabolic PDE has been established, the regularity of $f$ follows from standard arguments.
        Nevertheless, we provide a proof for completeness.
        Let $\zeta \in C^3_c(Q^k_1)$ be a cut-off function such that
        \begin{equation}
            \label{eq:cutoff}
            \zeta \equiv 1 \text{ on } Q^k_{1/2}, \qquad 0 \leq \zeta \leq 1, \qquad \norm{\zeta}_{C^3(Q^k_1)} \leq C(k).
        \end{equation}
        Then, for each $a = k+1,\dots,n$, $\tilde{f}^a \coloneqq \zeta f^a$ solves the initial-boundary value problem
        \begin{equation}
            \label{eq:IBP}
            \left\{
                \begin{aligned}
                    \partial_t\tilde{f}^a - g^{ij}\partial_{ij}\tilde{f}^a &= \tilde{U}^a &&\text{in } Q^k_1, \\
                    \tilde{f}^a &= 0 && \text{on } \partial B^k_1 \times (-1,0), \\
                    \tilde{f}^a(\cdot,-1) &= 0 && \text{in } B^k_1,
                \end{aligned}
            \right.
        \end{equation}
        where $\tilde{U}^a \coloneqq \zeta U^a + \partial_t\zeta f^a - 2g^{ij}\partial_i\zeta\partial_j f^a - g^{ij}\partial_{ij}\zeta f^a$.
        Moreover, by \cite[Theorem 7.17]{Lie96}, $\tilde{f}^a$ is the unique solution to \eqref{eq:IBP} in $W^{1,2}((-1,0);L^2(B^k_1)) \cap L^2((-1,0);W^{2,2}(B^k_1))$.

        We first prove Theorem \ref{thm:reg_Lpq}.

    \begin{proof}[Proof of Theorem \ref{thm:reg_Lpq}]
        Without loss of generality, we may assume $R = 1$ and $T = \R^k \times \{0\}$.
        Fix $a = k+1,\dots,n$.
        By assumption \ref{itm:Lpq_integrability}, we have
        \begin{equation*}
            \norm{U}_{L^{p,q}(Q^k_1)} \leq (1 + \norm{\nabla f}_{0,Q^k_1})\norm{u}_{L^{p,q}}.
        \end{equation*}
        Hence, by \cite[Theorem 2.3]{DHP07}, there exists a unique solution
        \begin{equation*}
            \bar{f}^a \in W^{1,q}((-1,0);L^p(B^k_1)) \cap L^q((-1,0);W^{2,p}(B^k_1))
        \end{equation*}
        to \eqref{eq:IBP}.
        The uniqueness of solutions to \eqref{eq:IBP} in $W^{1,2}((-1,0);L^2(B^k_1)) \cap L^2((-1,0);W^{2,2}(B^k_1))$ implies that
        \begin{equation*}
            \tilde{f}^a = \bar{f}^a \in W^{1,q}((-1,0);L^p(B^k_1)) \cap L^q((-1,0);W^{2,p}(B^k_1)).
        \end{equation*}
        The estimate \eqref{eq:Lpq_est} follows, for example, from \cite[Theorem 5.2]{DK18}.
        Since \cite[Theorem 5.2]{DK18} is stated for PDEs on the whole space-time, it is necessary to extend $\tilde{f}^a$, $g^{ij}$, and $\tilde{U}^a$ to $\R^k \times \R$.
        Specifically, we extend $\tilde{f}^a$ and $\tilde{U}^a$ by zero to $\R^k \times (-\infty,0]$, and extend $g^{ij}$ suitably to $\R^k \times (-\infty,0]$ so that it remains uniformly continuous and satisfies \eqref{eq:g_bound} and \eqref{eq:g_ellipticity} there.
        We then take their even extensions with respect to $t = 0$.
        The resulting functions satisfy $\partial_t\tilde{f}^a - g^{ij}\partial_{ij}\tilde{f} = \tilde{U}^a$ in $\R^k \times \R$, and thus \cite[Theorem 5.2]{DK18} can be applied.
    \end{proof}

    Finally, we turn to the proof of Theorem \ref{thm:reg_Holder}.

    \begin{proof}[Proof of Theorem \ref{thm:reg_Holder}]
        Without loss of generality, we may assume $R = 1$ and $T = \R^k \times \{0\}$.
        Fix $a = k+1,\dots,n$.
        Since $u$ and $\nabla f$ are $\alpha$-H\"{o}lder continuous, we have
        \begin{gather*}
            \norm{U}_{\alpha,Q^k_1} \leq C\norm{u}_{\alpha,Q^k_1}, \qquad
            [g^{ij}]_{\alpha,Q^k_1} \leq C[\nabla f]_{\alpha,Q^k_1},
        \end{gather*}
        where $C > 0$ is a constant depending only on $n$, $k$, $\alpha$, $\norm{\nabla f}_{0,Q^k_1}$, and $[f]_{1+\alpha,Q^k_1}$.
        Hence, by \cite[Theorem 5.14]{Lie96}, there exists a unique classical solution $\bar{f}^a \in C^{2,\alpha}(\overline{Q}^k_1)$ to \eqref{eq:IBP}.
        Moreover, the uniqueness of solutions to \eqref{eq:IBP} in $W^{1,2}((-1,0);L^2(B^k_1)) \cap L^2((-1,0);W^{2,2}(B^k_1))$ implies that $\tilde{f}^a = \bar{f}^a \in C^{2,\alpha}(\overline{Q}^k_1)$.
        Applying \cite[Theorem 4.9]{Lie96} to \eqref{eq:PDE}, we obtain \eqref{eq:Holder_est}.
    \end{proof}

    \bibliography{ref}
    \bibliographystyle{alpha}
\end{document}